\documentclass{amsart}
\usepackage[margin=1.4in]{geometry}

\usepackage{amsfonts,amsmath,amssymb,amsthm,amscd,amsxtra}
\usepackage{mathtools,mathptmx}
\usepackage{enumerate,epsfig,verbatim}
\usepackage{centernot}
\usepackage{todonotes}
\usepackage[T1]{fontenc}
\usepackage[usenames,dvipsnames]{pstricks}
\usepackage[mathscr]{eucal}
\usepackage[notcite,notref]{}
\usepackage{tikz}
\usetikzlibrary{graphs,graphs.standard,matrix, calc, arrows,quotes}

\usepackage[all,2cell,ps]{xy}
\usepackage[pagebackref]{hyperref}
\usepackage{nicefrac}
\usepackage{array}
\usepackage{xcolor}
\usepackage{color}
\usepackage[all]{xypic}

\usepackage[notcite,notref]{}
\usepackage{url}
\usepackage{datetime}

\newcommand{\up}[1]{{{}^{#1}\!}}
\SelectTips{cm}{}

\def\latex/{{\protect\LaTeX}}
\def\latexe/{{\protect\LaTeXe}}
\def\amslatex/{{\protect\AmS-\protect\LaTeX}}
\def\tex/{{\protect\TeX}}
\def\amstex/{{\protect\AmS-\protect\TeX}}
\def\bibtex/{{Bib\protect\TeX}}
\def\makeindx/{\textit{MakeIndex}}

\theoremstyle{plain} 

\newtheorem{thm}{Theorem}[section]
\newtheorem{prop}[thm]{Proposition}
\newtheorem{cor}[thm]{Corollary}

\theoremstyle{definition}
\newtheorem{chunk}[thm]{\hspace*{-1.065ex}\bf}

\newtheorem{eg}[thm]{Example}
\newtheorem{conj}[thm]{Conjecture}

\newtheorem{rmk}[thm]{Remark}

\theoremstyle{remark}

\newtheorem*{claim*}{Claim}

\newcommand{\fm}{\mathfrak{m}}
\newcommand{\fp}{\mathfrak{p}}
\newcommand{\fq}{\mathfrak{q}}
\newcommand{\F}{\mathbb{F}}

\DeclareMathOperator{\e}{\operatorname{{e}}}

\newcommand{\iso}{\cong}

\DeclareMathOperator{\id}{id}

\DeclareMathOperator{\Tor}{Tor}
\DeclareMathOperator{\Ext}{Ext}

\DeclareMathOperator{\Hom}{Hom}

\DeclareMathOperator{\Tr}{\textnormal{Tr}}

\DeclareMathOperator{\Gdim}{\textnormal{G-dim}}
\DeclareMathOperator{\len}{\textup{length}}

\DeclareMathOperator{\Ass}{Ass}

\DeclareMathOperator{\Spec}{Spec}

\DeclareMathOperator{\pd}{pd}

\DeclareMathOperator{\End}{End}
\DeclareMathOperator{\grade}{grade}

\DeclareMathOperator{\depth}{depth}

\DeclareMathOperator{\X}{\operatorname{\mathsf{X}}}

\newcommand{\Min}{\textup{Min}}

\newcommand{\Ann}{\textup{Ann}}

\newcommand{\bb}{\left[ \begin{smallmatrix}}
\newcommand{\eb}{\end{smallmatrix} \right]}

\def\urltilda{\kern -.15em\lower .7ex\hbox{\~{}}\kern
  .04em}\def\urldot{\kern -.10em.\kern -.10em}\def\urlhttp{http\kern
  -.10em\lower -.1ex\hbox{:}\kern -.12em\lower 0ex\hbox{/}\kern
  -.18em\lower 0ex\hbox{/}}

\begin{document}

\title[Two theorems on the vanishing of Ext]{Two theorems on the vanishing of Ext}


\author[Olgur Celikbas]{Olgur Celikbas}
\address{Olgur Celikbas\\
School of Mathematical and Data Sciences\\
West Virginia University\\
Morgantown, WV 26506-6310, USA}
\email{olgur.celikbas@math.wvu.edu}

\author[Souvik Dey]{Souvik Dey}
\address{Souvik Dey\\ 
Charles University, Faculty of Mathematics and Physics, Department of Algebra, Sokolovsk´a 83, 186 75 Praha, Czech Republic.} 
\email{souvik.dey@matfyz.cuni.cz}

\author[Toshinori Kobayashi]{Toshinori Kobayashi}
\address{Toshinori Kobayashi \\ School of Science and Technology, Meiji University, 1-1-1 Higashi-Mita, Tama-ku, Kawasaki-shi, Kanagawa 214-8571, Japan}
\email{tkobayashi@meiji.ac.jp}

\author[Hiroki Matsui]{Hiroki Matsui}
\address{Hiroki Matsui\\ Department of Mathematical Sciences,
Faculty of Science and Technology,
Tokushima University,
2-1 Minamijosanjima-cho, Tokushima 770-8506, JAPAN}
\email{hmatsui@tokushima-u.ac.jp}

\author[Arash Sadeghi]{Arash Sadeghi}
\address{Arash Sadeghi\\ Dokhaniat 49179-66686, Gorgan, IRAN}
\email{sadeghiarash61@gmail.com}

\thanks{Souvik Dey was partly supported by the Charles University Research Center program No. UNCE/SCI/022 and a grant GA CR 23-05148S from the Czech Science Foundation}
\thanks{Toshinori Kobayashi was partly supported by JSPS Grant-in-Aid for JSPS Fellows 21J00567}
\thanks{Hiroki Matsui was partly supported by JSPS Grant-in-Aid for Early-Career Scientists 22K13894}
\maketitle{}

\vspace{-0.2in}
\begin{center}
\small{\emph{To the memory of Shiro Goto}}
\end{center}

\begin{abstract} We prove two theorems on the vanishing of Ext over commutative Noetherian local rings. Our first theorem shows that there are no Burch ideals which are rigid over non-regular local domains. Our second theorem reformulates a conjecture of Huneke-Wiegand in terms of the vanishing of Ext, and highlights its relation with the celebrated Auslander-Reiten conjecture. We also discuss several consequences of our results, for example, about the rigidity of the Frobenius endomorphism in prime characteristic $p$ and a generalization of a result of Araya.
\end{abstract}

\section{Introduction}
Throughout, unless otherwise stated, $R$ denotes a commutative Noetherian local ring with unique maximal ideal $\fm$ and residue field $k$, and modules over $R$ are assumed to be finitely generated. 

An $R$-module $M$ is called \emph{rigid} if $\Ext^1_R(M,M)=0$. We say $M$ has \emph{rank} provided that there is an integer $r\geq 0$ such that $M_{\fp} \cong R_{\fp}^{\oplus r}$ for all $\fp \in \Ass(R)$, that is, for all associated primes $\fp$ of $R$.

In this paper we are concerned with the Gorenstein case of a long-standing conjecture of Huneke and Wiegand; see \cite[page 473]{HW1} and also \cite[8.6]{CeRo}. 

\begin{conj} \label{HWC} (Huneke-Wiegand) Let $R$ be a one-dimensional Gorenstein ring and let $M$ be a torsion-free $R$-module that has rank. If $M$ is rigid, then $M$ is free. \qed
\end{conj}

The rank and one-dimensional hypotheses are necessary for Conjecture \ref{HWC} which is wide open in general, even for complete intersection rings of codimension two; see, for example, \cite{CeRo}. On the other hand there are affirmative answers over several classes of rings and for quite a few examples of modules. For example, a result of Huneke-Wiegand \cite[3.1]{HW1}, in view of \cite[2.13]{CGTT}, implies that Conjecture \ref{HWC} is true over hypersurface rings. We refer the reader to \cite{Ce, CGTT,  CeRo, Micah, HWendo, CrRoSi, Lindo} for details and further examples concerning the conjecture.

The aim of this paper is twofold: we prove two theorems, namely Theorem \ref{thm1} and Theorem \ref{thm2}, and make progress on Conjecture \ref{HWC} from two different perspectives. As a byproduct, we also generalize a theorem of Araya which considers maximal Cohen-Macaulay modules over Gorenstein rings; see \ref{Tok}.


Dey and Kobayashi \cite[3.1]{SDTK} extended the definition of a Burch ideal \cite{LRT} to a Burch module: an $R$-module $M$ is said to be \emph{Burch} if it is an $R$-submodule of an $R$-module $T$ such that $\fm (M:_T \fm) \nsubseteq \fm M$. As Burch ideals are Burch modules, examples of Burch modules are abundant in the literature; see \ref{HWex}. Our first theorem establishes Conjecture \ref{HWC} for Burch modules and show that, over one-dimensional non-regular local domains, there are no ideals that are both Burch and rigid. In fact our result is more general and shows the following; see Theorem \ref{mainthm}.

\begin{thm} \label{thm1} Let $R$ be ring and let $M$ be a torsion-free $R$-module which has rank.  If $M$ is Burch and rigid, then $M$ is free, and $R$ is a field or a discrete valuation ring.  \qed
\end{thm}

If $R$ and $M$ are as in Conjecture \ref{HWC}, it follows that $M$ is rigid if and only if $M\otimes_RM^{\ast}$ is torsion-free, where $M^{\ast}=\Hom_R(M,R)$; see \ref{HJ}(i). Hence, as a consequence of Theorem \ref{thm1}, we obtain:

\begin{cor} \label{corthm3} Let $R$ be a one-dimensional Gorenstein ring and let $M$ be an $R$-module which has rank. Assume $M$ and $M\otimes_RM^{\ast}$ are both torsion-free. If $M$ is Burch, then $R$ is regular and $M$ is free.  \qed
\end{cor}

If $R$ is a one-dimensional F-finite Cohen-Macaulay local ring of prime characteristic $p$, and if $\up{e}{R}$ is rigid for some $e\gg 0$, then it follows from \cite[6.10]{HD2} that $R$ is a discrete valuation ring. As a consequence of Theorem \ref{thm1}, we are able to extend this result and prove the following corollary in section 2; see also \ref{HWex}(vi) and Remark \ref{rmkFrob}.

\begin{cor}\label{corFrob} Let $R$ be an equi-dimensional reduced ring of prime characteristic $p$ with $\depth(R)=1$. Assume $R$ is F-finite, for example, R is excellent and k is perfect. If $\up{e}{R}$ is rigid for some $e\gg 0$, then $R$ is a discrete valuation ring. \qed
\end{cor}

The commutative version of the celebrated Auslander-Reiten conjecture \cite{AuRe} claims that each $R$-module $M$ must be free if $\Ext^i_R(M,M)=\Ext^i_R(M,R)=0$ for all $i\geq 1$. It is known that, if Conjecture \ref{HWC} holds, then the Auslander-Reiten conjecture holds over each Gorenstein local domain (of arbitrary dimension); see \cite[8.6]{CeRo} for the details. The second theorem we prove generalizes this fact and shows that Conjecture \ref{HWC} implies a stronger conclusion over Gorenstein local domains:

\begin{thm} \label{thm2} The following conditions are equivalent:
\begin{enumerate}[\rm(a)]
\item Conjecture \ref{HWC} holds over each (one-dimensional Gorenstein local) domain.
\item Whenever $R$ is a $d$-dimensional Gorenstein domain and $M$ is a torsion-free $R$-module such that $\Ext^i_R(M,M)=0$ for all $i=1, \ldots, d$, it follows that $M$ is free. \qed
\end{enumerate} 
\end{thm}

In view of the foregoing discussion, Theorem \ref{thm2} yields the following observation:

\begin{cor} The Auslander-Reiten conjecture holds over each Gorenstein local domain (of arbitrary dimension) if the following condition holds: Whenever $R$ is a $d$-dimensional Gorenstein domain and $M$ is a torsion-free $R$-module such that $\Ext^i_R(M,M)=0$ for all $i=1, \ldots, d$, it follows that $M$ is free. \qed
\end{cor}

Let us note the following result concerning Theorem \ref{thm2}: if $R$ is a $d$-dimensional Gorenstein local ring with $d\geq 1$, $M$ is a torsion-free $R$-module such that $M_{\fp}$ is free over $R_{\fp}$ for all prime ideals $\fp$ of height at most one, and $\Ext_R^{i}(M, M) = 0$ for $i = 1, \ldots, d-1$, then $M$ is free; see Corollary~\ref{conse}.

The main ingredient of the proof of Theorem \ref{thm2} is Proposition \ref{AR1}. A byproduct of that proposition, along with a result of Kimura \cite{KK}, allows us to prove the following in the appendix.

\begin{prop} \label{corAson} Let $R$ be a ring such that $d=\depth(R)\geq 1$ and let $M$ be an $R$-module. Assume the following hold:
\begin{enumerate}[\rm(i)]
\item $\pd_{R_{\fp}}(M_{\fp})<\infty$ for all $\fp \in \Spec(R)-\{\fm\}$.
\item $M$ is torsion-free.
\item $\Gdim_R(M)<\infty$ 
\end{enumerate}
If $\Ext^{d-1}_R(M,M)=0$, then $\pd_R(M)\leq d-2$.
\end{prop}

Proposition \ref{corAson} was initially proved by Araya \cite{Ar} for the case where $R$ is Gorenstein and $M$ is maximal Cohen-Macaulay; see \ref{G} and \ref{Tok}. We also prove a variation of Proposition \ref{corAson} by replacing the finite Gorenstein dimension hypothesis on $M$ with the assumption of the vanishing of $\Ext^i_R(M,R)$ for all $i=d, \ldots, 2d+1$ under mild additional conditions; see Proposition \ref{mt} and Corollary \ref{corneww}.
%

\section{Proof of Theorem \ref{thm1}} 

In this section we prove Theorem \ref{thm1}. First we recall some examples of Burch ideals and modules, and prepare auxiliary results which play an important role in the proof of the theorem. 

\begin{chunk}\label{s2f3} Let $M$ be an $R$-module. Then $M$ is said to be a \emph{Burch} $R$-submodule of an $R$-module $T$ provided that $\fm (M:_T \fm) \nsubseteq \fm M$; if such a module $T$ exists, we simply call $M$ a Burch $R$-module; see \cite[3.1]{SDTK}. Note that, by definition, a Burch module is nonzero.
\end{chunk}

We need the following properties in the sequel:

\begin{chunk}\label{s2f33} Let $M$ be an $R$-module. 
\begin{enumerate}[\rm(i)] 
\item If $\depth_R(M)\geq 1$, then $M$ is Burch if and only if there is an $x \in \fm$ such that $x$ is a non zero-divisor on $M$ and $k$ is a direct summand of $M/xM$; see \cite[3.6 and 3.9]{SDTK}.
\item If $M$ is Burch and $\pd_R(M)<\infty$ (or $\id_R(M)<\infty$), then $R$ is regular; see \cite[3.19]{SDTK}. \qed
\end{enumerate}
 \end{chunk}

There are many examples of Burch ideals and modules. Here we record a few of them:

\begin{eg} If $R=k[\![t^4, t^5, t^6 ]\!]$ and $I=(t^{17}, t^{19}, t^{20})$, or $R=k[\![x,y ]\!]$ and $I=(x^5, x^3y, xy^3, y^5)$, then $I$ is Burch; see \cite[4.4 and 4.5]{CelKob} and also cf. Example \ref{excompare}.
\end{eg}

\begin{chunk} \label{HWex} Let $I$ be an ideal of $R$ and let $M$ be an $R$-module.
\begin{enumerate}[\rm(i)]
\item If $\depth(R)\geq 1$ and $I$ is integrally closed, then $I$ is Burch; see \cite[2.2(4)]{CelKob}.
\item If $R$ is not a field, $|k|=\infty$, $\depth(R/I)=0$, and $I$ is integrally closed, then $I$ is Burch; see \cite[4.6]{GS}. 
\item If $M$ is an $R$-submodule of an $R$-module $T$ such that  $(\fm M:_T \fm) \subseteq M$ and $\depth_R(T/M)=0$, then $M$ is a Burch submodule of $T$; see \cite[4.3]{SDTK}. 
\item If $\fm M \neq 0$, for example, if $\depth_R(M)\geq 1$, then $\fm M$ is Burch; see \cite[3.4]{SDTK}.
\item If $M$ is a Burch $R$-submodule of an $R$-module $T$, then $\big(M+xT[x]\big)_{(\fm, x)}$ is a Burch $S$-submodule of the $S$-module $T[x]_{(\fm, x)}$, where $S=R[x]_{(\fm, x)}$; see \cite[4.8(2)]{GS}. 
\item If $R$ is of prime characteristic $p$, $F$-finite, and $\depth(R)=1$, then $\up{e}{R}$ is Burch for all $e\gg 0$ (here $\up{e}{R}$ denotes the $R$-module whose underlying abelian group is $R$, and the $R$-action on $\up{e}{R}$ is given by $r \cdot x=r^{p^e}x$ for $r\in R$ and $x\in \up{e}{R}$); see \ref{s2f33}(i) and \cite[3.3]{YuTak}.
\end{enumerate}
\end{chunk} 
\begin{chunk} \label{Ul} A Cohen-Macaulay $R$-module $N$ is called \emph{Ulrich} \cite[2.1]{UlrichGTT} if the minimal number of generators $\mu_R(N)$ of $N$ equals the (Hilbert-Samuel) multiplicity $\e_R(N)$ of $N$. Here $$\e_R(N)= t! \cdot \lim_{n \to \infty} \dfrac{\len_R(N/\fm^n N)}{n^{t}} \text {, where } t=\dim_R(N).$$

Ulrich modules were initially defined in \cite{Ulrish} as maximally generated maximal Cohen-Macaulay modules. Such modules were studied extensively in the literature; see, for example, \cite{GotoUlrich} for details. 

Let $N$ be a one-dimensional Cohen-Macaulay $R$-module. It follows from \cite[4.8 and 5.2]{SDTK} that, if $N$ is Ulrich, or equivalently, $N\cong \fm N$, then $N$ is Burch. On the other hand, if $N$ is Burch, then it does not need to be Ulrich. For example, if $R$ is a one-dimensional Cohen-Macaulay ring which does not have minimal multiplicity, that is, $\e(R)>\mu_R(\fm)$, and $N=\fm$, then $N$ is Burch but not Ulrich. Similarly, if $R$ is a one-dimensional non-regular Cohen-Macaulay ring and $N=R\oplus \fm$, then $N$ is Burch, but it is not an Ulrich $R$-module (since $R$ is not Ulrich, or equivalently, since $R$ is not regular). \qed 
\end{chunk}

The following observation plays an important role for the proof of Theorem \ref{mainthm}.

\begin{chunk} \label{s2f1} Let $R$ be a Henselian ring, $N$ an indecomposable rigid $R$-module, and let $x\in \fm$ be a non zero-divisor on $N$. Then  $N/xN$ is an indecomposable $R$-module. Therefore, if there is a nonzero cyclic $R$-module which is a direct summand of $N/xN$ as an $R$-module, then $N$ is cyclic.
\end{chunk}

\begin{proof} Note, as $\Ext^1_R(N,N)=0$, the short exact sequence $0 \to N \xrightarrow{x} N \to N/xN \to 0$ gives the short exact sequence $0 \to \Hom_R(N,N) \xrightarrow{x} \Hom_R(N,N) \to \Hom_R(N, N/xN) \to 0$. This yields an isomorphism $\Hom_R(N, N/xN) \cong \End_R(N)/x\End_R(N)$. It follows, since $\Hom_R(N, N/xN) \cong \End_{R/xR}(N/xN)$, that there is a ring isomorphism $\End_{R/xR}(N/xN) \cong \End_R(N)/x\End_R(N)$. Note, since $R$ is Henselian and $N$ is indecomposable, we know that $\End_R(N)$ is a local ring; see \cite[1.8]{TheBook}. So we see that $\End_{R/xR}(N/xN)$ is a local ring. Thus $N/xN$ is an indecomposable $R$-module. 

If there is a nonzero cyclic $R$-module, say $L$, which is a direct summand of $N/xN$ as an $R$-module, then $N/xN \cong L$ so that $\mu_R(N)=\mu_R(N/xN)=\mu_R(L)=1$, that is, $N$ is cyclic, as claimed.
\end{proof}

\begin{rmk} Let us note that the rigidity hypothesis is needed in \ref{s2f1}: Let $R$ be a non-regular domain. Then each nonzero proper ideal of $R$ is indecomposable. However, if $x\in \fm-\fm^2$, then the $R$-module $\fm/x\fm$ is decomposable; see, for example, \cite[5.3]{TakSyz} or \cite[7.11]{HD2}. Note that $\fm$ is not a rigid $R$-module since $R$ is not regular. \qed
\end{rmk}

Next we prepare a proposition to facilitate the proof of Theorem \ref{mainthm}. Note that $\grade_R(N)$ denotes the common length of a maximal $R$-regular sequence in $\Ann_R(N)$ of the $R$-module $N$; see \cite[1.2.6]{BH}.

\begin{prop} \label{mainprop} Let $R$ be a Henselian ring and let $M$ be an $R$-module. Assume $M$ is Burch, rigid, and $\depth_R(M)\geq 1$. Then the following hold:
\begin{enumerate}[\rm(i)]
\item $M$ has a direct summand $N$, where $N$ is Cohen-Macaulay, cyclic, rigid, $\dim_R(N)=1$, $\fm N \cong N$, and $\grade(N)=0$.
\item If $M$ is torsion-free, then $\depth(R)\leq 1$.
\end{enumerate}
\end{prop}

\begin{proof} Note that, since $M$ is Noetherian, we can write $M \cong \bigoplus_{i=1}^{n} M_i$ for some nonzero indecomposable $R$-modules $M_1, \ldots, M_n$. As $M$ is Burch and $\depth_R(M)\geq 1$, there is an $x \in \fm$ such that $x$ is a non zero-divisor on $M$ and $k$ is a direct summand of $M/xM$; see \ref{s2f33}(i). Therefore we have $M/xM \cong \bigoplus_{i=1}^{n} M_i/xM_i$. As $k$ is a direct summand of $M/xM$ and $k$ is indecomposable, we see by the Krull–Schmidt theorem \cite[1.8]{TheBook} that $k$ is a direct summand of $M_j / xM_j$ for some $j$, where $1\leq j \leq n$. Set $N=M_j$. 

As $M$ is rigid and $N$ is a direct summand of $M$, we see that $N$ is rigid too. Moreover $x$ is a non zero-divisor on $N$. So $N$ is cyclic by \ref{s2f1}. This implies that $N/xN$ is a cyclic $R$-module and therefore it is an indecomposable $R$-module. As $k$ is a direct summand of $N/xN$, we conclude that $k\cong N/xN$. This shows that $N$ is a Cohen-Macaulay $R$-module with $\dim_R(N)=1$. Moreover $\fm (N/xN) =0$ so that $\fm N = xN \cong N$. This proves part (i).

Next assume $M$ is torsion-free. Then $N$ is also torsion-free. Therefore, if $\grade_R(N)>0$, then $N^{\ast}=0$ so that $N=0$; see \cite[1.2.3(b)]{BH}. This proves $\grade_R(N)=0$.

Let $\fp \in \Ass_R(N)$. Then, since $N$ is torsion-free, there is an $\fq \in \Ass(R)$ such that $\fp \subseteq \fq$. Therefore we have $\depth(R)\leq \dim(R/\fq) \leq \dim(R/\fp)=\dim_R(N)=1$; see \cite[1.2.13 and 2.1.2]{BH}.
\end{proof}

Proposition \ref{mainprop}(i), in view of \ref{Ul}, yields the following relation between Burch and Ulrich modules:

\begin{cor} \label{U} Let $R$ be a Henselian ring and let $M$ be an $R$-module such that $M$ is Burch, rigid, and $\depth_R(M)\geq 1$. Then $M$ has a direct summand $N$, where $N$ is a cyclic Ulrich $R$-module with $\dim_R(N)=1$. \qed
\end{cor}

We need a few more observations prior to giving a proof of the main theorem of this section. 

\begin{chunk} \label{preserve} Let $M$ be an $R$-module. Let  $\widehat{(-)}$ denote the $\fm$-adic completion functor.
\begin{enumerate}[\rm(a)]
\item Assume $M$ is Burch, that is, $M$ is an $R$-submodule of an $R$-module $T$ such that $\fm (M:_T \fm) \nsubseteq \fm M$. Set $X=\fm (M:_T \fm)$ and $Y= \fm M$. If $\widehat{X} \subseteq \widehat{Y}$, then $\widehat{X} \cap \widehat{Y}=\widehat{X \cap Y}=\widehat{X}$ so that $X\cap Y=X$, that is, $X\ \subseteq Y$. Hence $\widehat{X} \nsubseteq \widehat{Y}$ so that $\widehat{M}$ is Burch over $\widehat{R}$.

\item If $M$ is torsion-free, or equivalently, if $\bigcup \Ass_R(M) \subseteq \bigcup \Ass(R)$, then we do not know whether or not $\widehat{M}$ must be torsion-free over $\widehat{R}$. On the other hand we discuss two affirmative cases:
\begin{enumerate}[\rm(i)]
\item Assume $M$ is torsion-free and generically free. Then $M$ is an $R$-submodule of a free $R$-module; see \cite[3.5(c)]{EG}. As the same holds for $\widehat{M}$ over $\widehat{R}$, it follows that $\widehat{M}$ is torsion-free over $\widehat{R}$.
\item In general, if $\Ass_R(M)\subseteq \Ass(R)$, then $\Ass_{\widehat{R}}(\widehat{M})\subseteq \Ass(\widehat{R})$; see \cite[23.2]{Mat}. Hence, if $M$ is torsion-free and $\Ass(R)=\Min(R)$, then $\Ass_R(M) \subseteq \Min(R)$ so that $\widehat{M}$ is torsion-free over $\widehat{R}$; see, for example, \cite[3.8]{YYNE}. 
\end{enumerate}
%

\item Assume $M$ has rank, say $r$. Let $\fq\in \Ass(\widehat{R})$ and set $\fp=\fq \cap R$. Then $\depth(R_{\fp})\leq \depth(\widehat{R}_{\fq})=0$; see \cite[1.2.16]{BH}. Hence $\widehat{M}_{\fq} \cong M_{\fp} \otimes_{R_{\fp}} \widehat{R}_{\fq} \cong \widehat{R}^{\oplus r}_{\fq}$. This shows that $\widehat{M}$ has rank $r$ over $\widehat{R}$. \qed
\end{enumerate}
\end{chunk}

The next result stems from the arguments of Levin-Vascencelos \cite{LV}.

\begin{chunk} \label{sonekle} Let $M$ and $N$ be $R$-modules such that $M$ has rank and $\grade_R(N)=0$. If $\Ext^t_R(M, \fm N)=0$, then $\pd_R(M)<t$; see \cite[2.6(6)]{SDTK}. \qed
\end{chunk}

Next is our main result in this section:

\begin{thm} \label{mainthm} Let $M$ be an $R$-module which is Burch, rigid, and torsion-free. 
\begin{enumerate}[\rm(i)]
\item If $\Ass(R)=\Min(R)$, then $\depth(R)\leq 1$, that is, $R$ is either Artinian or $\depth(R)=1$.
\item If $M$ has rank, then $M$ is free, and $R$ is either a field or a discrete valuation ring.
\end{enumerate}
\end{thm}

\begin{proof} If $\Ass(R)=\Min(R)$ or $M$ has rank, then $\widehat{M}$ is torsion-free over $\widehat{R}$; see \ref{preserve}(b)(iii). Therefore, by considering $\widehat{M}$ over $\widehat{R}$, we may assume $R$ is Henselian and $M$ is a Burch, rigid, and torsion-free $R$-module; see \ref{preserve}(a) and \cite[1.9 and 1.10]{TheBook}. Hence Proposition \ref{mainprop} shows that  $\depth(R)\leq 1$ and $M$ has a direct summand $N$ such that $\fm N \cong N$ and $\grade_R(N)=0$.

As $M$ is rigid, we have that $\Ext^1_R(M,N)=\Ext^1_R(M, \fm N)=0$. Hence, if $M$ has rank, then \ref{sonekle} implies that $M$ is free. Thus we conclude from \ref{s2f33}(ii) that $R$ is a field or a discrete valuation ring.
\end{proof}

The rank hypothesis in Theorem \ref{mainthm}(ii) is necessary; we can see this by the following example.

\begin{eg} Let $R=k[\![x,y]\!]/(xy)$ and $M=R/(x)$. Then $M \cong k[\![y]\!]$. Hence $\depth_R(M)=1$, $y$ is a non zero-divisor on $M$, and $k\cong M/yM$. Therefore $M$ is a Burch $R$-module; see \ref{s2f33}(ii). Moreover $M$ is rigid, but $M$ does not have rank since $M_{(x)}\neq 0=M_{(y)}$. \qed
\end{eg}

Next we prove Corollary \ref{corFrob} advertised in the introduction:

\begin{proof} [Proof of Corollary \ref{corFrob}] Note that $\up{e}{R}$ is a torsion-free $R$-module. Note also that, since $R$ is reduced, $R_{\fp}$ is a field for each $\fp \in \Ass(R)$. Hence $\up{e}{R}$ has rank for $e\gg 0$; see \cite[2.3 and 2.4]{Kunz}. Hence the claim follows from \ref{HWex}(vi) and Theorem \ref{mainthm}(ii).
\end{proof}

\begin{rmk}\label{rmkFrob} As mentioned in the introduction, if $R$ is a one-dimensional F-finite Cohen-Macaulay local ring of prime characteristic $p$, and if $\up{e}{R}$ is rigid for some $e\gg 0$, then it follows from \cite[6.10]{HD2} that $R$ is regular. The proof of \cite[6.10]{HD2} relies upon a result of Koh-Lee \cite[2.6]{KahMah}: If $R$ is as before and if $\Tor_n^R(M,\up{e}{R})=0$ for some $R$-module $M$ and $n\geq 1$, then $\pd_R(M)<\infty$; see also \cite[A.3]{HD2}. Let us note that, under the setup of Corollary \ref{corFrob} we do not know whether or not $\up{e}{R}$ enjoys a similar property. For example, if $R=\F_p[\![x^4, x^3y, xy^3, y^4]\!]$, then $R$ is an $F$-finite local domain of depth one and dimension two;  in that case \cite[6.10]{HD2} does not apply, but Corollary \ref{corFrob} shows that $\up{e}{R}$ is not rigid for all $e\gg 0$.\qed
\end{rmk}

Recall that, over a one-dimensional Gorenstein domain $R$, Conjecture \ref{HWC} is equivalent to the following statement: if $M$ is a torsion-free $R$-module such that $M\otimes_RM^{\ast}$ is torsion-free, then $M$ is free. To address the conjecture, O.~Celikbas and Kobayashi \cite{CelKob} proved the following result; see also \cite[7.4(2)]{SDTK} for a similar result for modules.

\begin{chunk}(\cite[1.3 and 3.8]{CelKob}) \label{CeTos} Let $R$ be a one-dimensional local domain which is not regular, and let $I$ and $J$ be ideals of $R$. If $0\neq I \subseteq \fm J$ and $(I:_RJ)=(\fm I:_R \fm J)$, then $I$ is Burch and $I\otimes_R I^{\ast}$ has nonzero torsion. \qed
\end{chunk}

In passing let us note that there are examples of ideals $I$ over one dimensional non-regular local rings $R$ such that $0\neq I \subseteq \fm J$ and $(I:_RJ)=(\fm I:_R \fm J)$ for some proper ideal $J$ of $R$; see, for example, \cite[4.3]{CelKob}. On the other hand, there can be Burch ideals $I$ for which this equality does not hold for any $J$ with $0\neq I \subseteq \fm J$. Next we give such an example. This example indicates that Theorem \ref{thm1} is not subsumed by the result stated in \ref{CeTos}.

\begin{eg} \label{excompare} Let $R=k[\![t^4, t^5, t^6 ]\!]$ and $I=(t^5, t^8)$. Then $\fm^2=(t^8, t^9, t^{10}, t^{11})$ and $\fm I=(t^9, t^{10}, t^{11}, t^{12})$. 

Claim 1: $I$ is Burch, that is, $\fm (I:_R \fm) \nsubseteq \fm I$.\\
Proof of Claim 1: As $\fm^2 \subseteq I$, it follows that $\fm \subseteq (I:_R \fm)$ and so $\fm=(I:_R \fm)$. Thus $\fm (I:_R \fm)=\fm^2  \nsubseteq \fm I$ as $t^8 \in \fm^2$ but $t^8 \notin \fm I$.
 
Claim 2: For each ideal $J$ of $R$ such that $0\neq I \subseteq \fm J$, then $(I:_RJ)\neq (\fm I:_R \fm J)$.\\
Proof of Claim 2: Let $J$ be an ideal of $R$ such that $I \subseteq \fm J$. If $J \neq R$, then $I \subseteq \fm^2$ since $J\subseteq \fm$, but that is not true because $t^5 \notin \fm^2$. So $J=R$, and we need to justify $I=(I:_RR)\neq (\fm I:_R \fm)$, or equivalently, $(\fm I:_R \fm) \nsubseteq I$. We can see this as follows: $t^6 \notin I$, but on the other hand $t^6 \in (\fm I:_R \fm)$ since $t^6 \fm =(t^{10}, t^{11}, t^{12}) \subseteq \fm I$.  \qed
\end{eg}

It is known that Conjecture \ref{HWC} holds for ideals over rings of the form $k[\![t^a, t^{a+1}, \ldots, t^{2a-2} ]\!]$ if $k$ is a field and $a\geq 3$; see \cite[1.6]{HWendo}. Hence the fact that $I$ in Example \ref{excompare} is not rigid can be deduced by \cite[1.6]{HWendo}, or by Theorem \ref{thm1} since it is a Burch ideal.


\section{Proof of Theorem \ref{thm2}}
 
The main purpose of this section is to prove Theorem \ref{thm2} and formulate Conjecture \ref{HWC} in terms of the vanishing of Ext over Gorenstein local domains. Our argument shows that Conjecture \ref{HWC} is indeed equivalent to a stronger version of a celebrated conjecture of Auslander and Reiten \cite{AuRe}; see the paragraph preceding Theorem \ref{thm2}.  

The proof of Theorem \ref{thm2} requires some preparation. For that we prove two propositions, each of which deals with modules that have finite Gorenstein dimension.

\begin{chunk} (\cite{AuBr}) \label{G} An $R$-module $M$ is said to be \emph{totally reflexive} if the natural map $M \to M^{\ast\ast}$ is bijective and $\Ext^i_R(M,R) = 0 = \Ext^i_R(M^{\ast},R)$ for all $i\geq 1$. If $M\neq 0$, then the infimum of $n$ for which there exists an exact sequence $0\to X_n \to \cdots \to X_0 \to M \to 0$, where each $X_i$ is totally reflexive, is called the \emph{Gorenstein dimension} of $M$. If $M$ has Gorenstein dimension $n$, we write $\Gdim_R(M) = n$. Therefore, $M$ is totally reflexive if and only if $\Gdim_R(M)\leq 0$, where it follows by convention that $\Gdim_R(0)=-\infty$. 

In the rest of the paper we use the following properties freely.
\begin{enumerate}[\rm(i)]
\item $R$ is Gorenstein if and only if $\Gdim_R(M)<\infty$ for each $R$-module $M$. 
\item $\Gdim_R(M) \leq \pd_R(M)$, and $\Gdim_R(M)=\pd_R(M)$ if $\pd_R(M)<\infty$.
\item $\Gdim_R(M)+\depth_R(M)=\depth(R)$ if $M\neq 0$ and $\Gdim_R(M)<\infty$. 
\item $\Gdim_{R_{\fp}}(M_{\fp})\leq \Gdim_R(M)$ for all $\fp \in \Spec(R)$. \qed
\end{enumerate}
\end{chunk}

The next result is key for our argument.

\begin{prop} \label{AR1} Let $M$ be an $R$-module. Assume $d=\depth(R)\geq 2$ and the following:
\begin{enumerate}[\rm(a)]
\item $\Gdim_R(M)\leq d-1$.
\item $\pd_{R_{\fp}}(M _{\fp})<\infty$ for each $\fp \in \Spec(R)-\{\fm\}$. 
\end{enumerate}
Then there exists an exact sequence of $R$-modules $0 \to Y \to X \to M \to 0$ such that the following hold:
\begin{enumerate}[\rm(i)]
\item $X$ is totally reflexive and locally free on $\Spec(R)-\{\fm\}$.
\item $\pd_R(Y)\leq d-2$.
\item $\Ext^d_R(M,M) \cong \Ext^d_R(X,X)$.
\item If $\Ext^{d-1}_R(M,M)=0$, then $\Ext^{d-1}_R(X,X)=0$.
\end{enumerate} 
\end{prop}

\begin{proof} It follows, since $\Gdim_R(M)<\infty$, that there is a short exact sequence of $R$-modules
\begin{equation}\tag{\ref{AR1}.1}
0 \to Y \to X \to M \to 0,
\end{equation}
where $\pd_R(Y)<\infty$ and $X$ is totally reflexive; see \cite[1.1]{AuBu}. As $\pd_{R_{\fp}}(M _{\fp})<\infty$ for each $\fp \in \Spec(R)-\{\fm\}$, it follows that $X$ is locally free on $\Spec(R)-\{\fm\}$. Note also $\depth_R(M)\geq 1$ since $\Gdim_R(M)\leq d-1$.

If $\depth_R(M)\geq d$, then the depth lemma applied to the exact sequence $0 \to Y \to X \to M \to 0$ shows that $\depth_R(Y)\geq d$; so, in this case, $Y$ is free since $\pd_R(Y)<\infty$. On the other hand, if $\depth_R(M)\leq d-1$, then, since $\depth_R(M)\geq 1$, we can use the depth lemma once more along with the exact sequence (\ref{AR1}.1) and see that $\depth_R(Y)=\depth_R(M)+1\geq 2$. Hence, regardless of what the depth of $M$ is, we have that $\pd_R(Y)\leq d-2$.

The exact sequence (\ref{AR1}.1) yields the following exact sequences for each $i\geq 1$:
\begin{equation}\tag{\ref{AR1}.1}
\Ext^i_R(X,Y) \to \Ext^i_R(X,X) \to \Ext^i_R(X,M)  \to \Ext^{i+1}_R(X,Y),
\end{equation}
and
\begin{equation}\tag{\ref{AR1}.2}
 \Ext^i_R(M,M) \to  \Ext^i_R(X,M) \to \Ext^{i}_R(Y,M) \to  \Ext^{i+1}_R(M,M) \to  \Ext^{i+1}_R(X,M) \to \Ext^{i+1}_R(Y,M).
\end{equation}

Note that $\Ext^j_R(X,Y)=0$ for all $j\geq 1$ since $\pd_R(Y)<\infty$ and $X$ is totally reflexive. So (\ref{AR1}.1) gives:
\begin{equation}\tag{\ref{AR1}.3}
\Ext^i_R(X,X) \cong \Ext^i_R(X,M) \text{ for each } i\geq 1.
\end{equation}

Note that, since $\pd_R(Y)\leq d-2$, we have $\Ext^{d-1}_R(Y,M)=\Ext^{d}_R(Y,M)=0$. Hence, setting $i=d-1$ in (\ref{AR1}.2), we obtain from 
(\ref{AR1}.2) and (\ref{AR1}.3) that:
\begin{equation}\tag{\ref{AR1}.4}
\Ext^d_R(M,M) \cong \Ext^d_R(X,M) \cong \Ext^d_R(X,X). 
\end{equation}

Now assume  $\Ext^{d-1}_R(M,M)=0$. Then, as $\Ext^{d-1}_R(Y,M)=0$, we see from (\ref{AR1}.2) that $\Ext^{d-1}_R(X,M)$ vanishes. So (\ref{AR1}.3) shows that 
$\Ext^{d-1}_R(X,X)=0$, as required. 
\end{proof}

We recall a few facts for the proof of Proposition \ref{mainthmsec3}.

\begin{chunk}
Recall that an $R$-module $M$ is said to satisfy Serre's condition $(S_n)$ for some integer $n\geq 0$ provided that $\depth_{R_{\fp}}(M_{\fp}) \geq \min\{n, \dim(R_{\fp}) \}$ for all $\fp \in \Spec(R)$; see \cite[page 3]{EG} (note that there are different versions of Serre's condition; cf. \cite[page 63]{BH}).

If $R$ is Gorenstein, then $M$ satisfies $(S_2)$ if and only if $M$ is reflexive if and only if $M$ is a second syzygy module; see \cite[3.6]{EG}
\end{chunk}

\begin{chunk} \label{mat} If $M\neq 0$ and $\pd_R(M)=n<\infty$, then $\Ext^n_R(M,M)\neq 0$; see \cite[Lemma 1(iii), page 154]{Mat}.
\end{chunk}


\begin{chunk} \label{Mat} \label{GT} \label{HJ} Let $R$ be a $d$-dimensional Cohen-Macaulay ring with canonical module $\omega_R$, and let $M$ and $N$ be $R$-modules. If $N$ is a maximal Cohen-Macaulay $R$-module such that $M$ or $N^\dagger$ is locally free on $\Spec(R)-\{\fm\}$, then $\Ext_{R}^i(M,N)\cong \Ext_{R}^i(N^\dagger\otimes_R M,\omega_R)$ for all $i=0, \ldots, d$, where $N^\dagger=\Hom_R(N,\omega_R)$; see \cite[2.3]{GoTak}.
\end{chunk}


\begin{prop} \label{mainthmsec3} Let $R$ be a $d$-dimensional Cohen-Macaulay local ring with a canonical module $\omega_R$. Consider the following two conditions on $R$:
\begin{enumerate}[\rm(i)]
\item Given a totally reflexive $R$-module $M$, if $M$ has rank and $M \otimes_R M^\dagger$ satisfies $(S_2)$, then $M$ is free.
\item Given a torsion-free $R$-module $N$, if $N$ has rank, $\Gdim_R(N)<\infty$, and $\Ext_R^i(N, N) = 0$ for all $i=1, \ldots, d$, then $N$ is free.
\end{enumerate} 
If $R_{\fp}$ satisfies the condition in part (i) for all $\fp \in \Spec(R)$, then $R$ satisfies the condition in part (ii).
\end{prop}

\begin{proof} Assume $R_{\fp}$ satisfies the condition in part (i) for all $\fp \in \Spec(R)$. Let $N$ be a torsion-free $R$-module with rank such that $\Gdim_R(N)<\infty$ and $\Ext_R^i(N, N) = 0$ for all $i=1, \ldots, d$. We proceed by induction on $d$ to prove that $N$ is free.

If $d=0$, then $N$ is free since it has rank. So we assume $d=1$. Then, since $N$ is totally reflexive and locally free on $\Spec(R)-\{\fm\}$, it follows from \ref{GT} that $\Ext^1_R(N,N)\cong \Ext_R^1(N\otimes_RN^{\dagger}, \omega_R)$. Therefore $\Ext_R^1(N\otimes_RN^{\dagger}, \omega_R)=0$. Hence $N \otimes_R N^\dagger$ is maximal Cohen-Macaulay and so it satisfies $(S_2)$. Consequently $N$ is free by the hypothesis.

Next assume $d\geq 2$. Then, by using the induction hypothesis, we may assume that $N$ is locally free on $\Spec(R)-\{\fm\}$. It follows from Proposition \ref{AR1} that there is an exact sequence $0 \to Y \to X \to N \to 0$,  where $X$ is totally reflexive and $\Ext^d_R(X,X)=\Ext^{d-1}_R(X,X)=0$. Note, by \ref{GT}, we have that $$\Ext^{i}_R(X\otimes_RX^{\dagger},\omega_R) \cong \Ext^{i}_R(X,X) \text{ for } i=d-1 \text{ and } i=d.$$ Hence \cite[3.5.11(b)]{BH} implies that $\depth_R(X\otimes_R X^{\dagger})\geq 2$. As $X$ is locally free on $\Spec(R)-\{\fm\}$, we see that $X\otimes_RX^{\dagger}$ satisfies $(S_2)$. So $X$ is free by the hypothesis (i). Consequently, as $\pd_R(Y)<\infty$, it follows that $\pd_R(N)<\infty$. Also, as $\Ext_R^i(N, N) = 0$ for all $i=1, \ldots, d$, we conclude by \ref{mat} that $N$ is free.
\end{proof}

If $R$ is a Noetherian integrally closed domain (not necessarily local) and $M$ is a torsion-free $R$-module such that $M\otimes_RM^{\ast}$ is reflexive, then Auslander \cite[3.3]{Au} proved that $M$ must be projective. In the local case, his argument in fact establishes the following result.


\begin{chunk} (Auslander \cite[3.3]{Au}; see also \cite[8.6]{CeRo}) \label{Aus} Let $R$ be a local ring satisfying $(S_2)$ and let $M$ be a torsion-free $R$-module such that $M\otimes_RM^{\ast}$ is reflexive. If $M_{\fp}$ is a free $R_{\fp}$-module for each prime ideal $\fp$ of $R$ of height at most one, then $M$ is free. \qed
\end{chunk}

We are now ready to give a proof of Theorem \ref{thm2}:

\begin{proof}[Proof of Theorem \ref{thm2}] Assume first condition (b) holds. Let $R$ be a one-dimensional Gorenstein local domain and let $N$ be a torsion-free $R$-module such that $N\otimes_RN^{\ast}$ is torsion-free. Then it follows from \ref{HJ} that $\Ext^1_R(N,N)=0$. Therefore, by our assumption, $N$ is free. This shows that Conjecture \ref{HWC} holds over each Gorenstein domain.

Next assume condition (a) holds, namely, Conjecture \ref{HWC} holds over each Gorenstein local domain. Let $R$ be a $d$-dimensional Gorenstein local domain, and let $M$ be a torsion-free $R$-module such that $\Ext^i_R(M,M)=0$ for all $i=1, \ldots, d$. We claim that $M$ is free.

Note that $\Gdim_R(M)<\infty$ as $R$ is Gorenstein. Therefore $M$ has all the properties stated in part (ii) of Proposition \ref{mainthmsec3}. Consequently, to show $M$ is free, it suffices to prove that $R_{\fp}$ satisfies the condition of part (i) of Proposition \ref{mainthmsec3} for all $\fp \in \Spec(R)$. 

For that let $\fp \in \Spec(R)$ and set $S=R_{\fp}$. Note that each module over $S$ has rank since $S$ is a domain. Let $X$ be a totally reflexive $S$-module such that $X \otimes_S X^{\ast}$ satisfies $(S_2)$. Then, since $S$ is Gorenstein, $X$ is a maximal Cohen-Macaulay $S$-module and $X \otimes_S X^{\ast}$ is reflexive. It suffices to show that $X$ is a free $S$-module. 

If $\dim(S)\leq 0$, then $S$ is a field and hence $X$ is free. If $\dim(S)=1$, then $X$ is free because it is assumed that Conjecture \ref{HWC} holds over each one-dimensional local Gorenstein domain and $S$ is such a ring. Next suppose $\dim(S)\geq 2$. Then, by the induction hypothesis on $\dim(S)$, we conclude that  $X_{\fq}$ is a free $S_{\fq}$-module for each prime ideal $\fq$ of $S$ of height at most one. So \ref{Aus} shows that $X$ is free.
\end{proof}

\appendix
\section{On a result of Araya}

In this appendix we state some byproducts of our argument that are not directly related to Conjecture \ref{HWC}. In particular, we prove Proposition \ref{corAson} which generalizes the following result of Araya:

\begin{chunk} (\cite[Corollary 10]{Ar}) \label{Tok} Let $R$ be a $d$-dimensional Gorenstein local ring, where $d\geq 1$, and let $M$ be an $R$-module. Assume the following hold: \begin{enumerate}[\rm(i)]
\item $M$ is locally free on $\Spec(R)-\{\fm\}$.
\item $M$ is maximal Cohen-Macaulay.
\item $\Ext^{d-1}_R(M,M)=0$. 
\end{enumerate}
Then $M$ is free. \qed
\end{chunk}

Araya's result \cite{Ar} shows that the Auslander-Reiten conjecture is straightforward over Gorenstein isolated singularities. It also implies that the conjecture holds over Gorenstein normal domains, a result initially proved by Huneke and Jorgensen \cite[5.9]{HJ}. Subsequently Kimura \cite{KK} removed the Gorenstein hypothesis, and proved that the Auslander-Reiten conjecture holds over arbitrary normal domains. In his paper Kimura also proved:

\begin{chunk} (\cite[2.8 and 2.10(3)]{KK}) \label{KK} Let $R$ be a local ring such that $d=\depth(R)\geq 1$, and let $X$ be an $R$-module such that $X_{\fp}$ is a free $R_{\fp}$-module for all  $\fp \in \Spec(R)-\{\fm\}$  and $\Ext^{d-1}_R(X,X)=0$. 
Then $X$ is free provided that at least one of the following holds:
\begin{enumerate}[\rm(i)]
\item $\Ext^{i}_R(X,R)=0$ for all $i=1, \ldots, 2d+1$.
\item $R$ satisfies $(S_2)$, $\Ext^{i}_R(X,R)=0$ for all $i=1, \ldots u$, and $\Ext^{i}_R(\Tr_R X,R)=0$ for all $i=1, \ldots v$ for some integers $u\geq 0$ and $v\geq 0$ such that $u+v=2d+1$. \qed
\end{enumerate}
\end{chunk}

Next we use Proposition \ref{AR1} and \ref{KK}(i), and give a proof for Proposition \ref{corAson} which is advertised in the introduction:

\begin{proof}[Proof of Proposition \ref{corAson}]  We may assume $d\geq 2$ and $M\neq 0$. As $M$ is torsion-free, it follows that $\depth_R(M)\geq 1$ and hence $\Gdim_R(M)\leq d-1$. Then Proposition \ref{AR1} gives an exact sequence $0 \to Y \to X \to M \to 0$, where $X$ is totally reflexive, $X$ is locally free on $\Spec(R)-\{\fm\}$, and $\Ext^{d-1}_R(X,X)=0$. Thus \ref{KK} implies that $X$ is free. Hence $\pd_R(M)\leq d-1$. Moreover, since $\Ext^{d-1}_R(M,M)=0$, we see from \ref{mat} that $\pd_R(M)\neq d-1$.
\end{proof}

In view of \ref{Mat}, a consequence of Proposition \ref{corAson} is:

\begin{cor} \label{corson1} Let $R$ be a $d$-dimensional local Gorenstein ring, with $d\geq 1$, and let $I$ be an $\fm$-primary ideal of $R$. Then $\Ext^{d-1}_R(I,I)\neq 0$. \qed
\end{cor}

In general we do not know whether or not $M$ must be free if $R$ is a $d$-dimensional Gorenstein local domain and $M$ is a torsion-free $R$-module such that $\Ext^i_R(M,M)=0$ for all $i=1, \ldots, d$; this is the context of Theorem~\ref{thm2}. However, if $M$ is locally free on the set $\X^1(R)$ of all prime ideals $\fp$ of $R$ with $\dim(R_{\fp})\leq 1$, then the vanishing of $\Ext^i_R(M,M)$ for all $i=1, \ldots, d-1$ is sufficient to conclude that $M$ is free. We record this observation which also extends \cite[1.6]{ACST} for the case where $n=1$.

\begin{cor} \label{conse} Let $R$ be a $d$-dimensional local ring satisfying $(S_1)$, where $d\geq 1$, and let $M$ be an $R$-module. Assume the following hold:
\begin{enumerate}[\rm(i)]
\item $M_{\fp}$ is free over $R_{\fp}$ for all $\fp \in \X^1(R)$.
\item $M$ is torsion-free. 
\item $\Ext_R^{i}(M, M) = 0$ for $i = 1, \ldots, d-1$.
\end{enumerate} 
If $\Gdim_R(M)<\infty$, then $M$ is free.
\end{cor}
\begin{proof} There is nothing to prove if $d=1$. Assume $d\geq 2$. Then $M_{\fp}$ is torsion-free over $R_{\fp}$ since $R$ satisfies $(S_1)$ \cite[3.8]{YYNE}. Hence, by the induction hypothesis on $d$, we have that $M_{\fp}$ is free over $R_{\fp}$ for all $\fp \in \Spec(R)-\{\fm\}$. So Proposition \ref{corAson} shows that $\pd_R(M)\leq \depth(R)-2$. As $\Ext_R^{i}(M, M) = 0$ for $i = 1, \ldots, d-1$, we conclude that $M$ is free; see \ref{mat}.
\end{proof}

We prove an analog of Proposition \ref{corAson} by replacing the finite Gorenstein dimension hypothesis on $M$ with the assumption of the vanishing of $\Ext^i_R(M,R)$ for $i=1, \ldots, d$ if $R$ satisfies $(S_2)$ and $M$ is locally free on $\Spec(R)-\{\fm\}$; see Corollary \ref{corneww}. This follows from a more general result which we prove as Proposition \ref{mt}; see also \cite[3.14]{HN}. In the proof we use \ref{KK}(ii) and the following auxiliary result:


\begin{chunk}\label{l1} Let $R$ be a local ring, $M$ and $N$ be $R$-modules, and let $n\geq 1$. Assume $\Ext^{n+1}_R(M,R)=0$. Then the canonical map $\Ext^{n}_R(M,N)\to \Ext^{n}_R(\Omega M,\Omega N)$ induced by the syzygy sequence is surjective. Moreover, if $\Ext^{n}_R(M,R)=0$, then $\Ext^{n}_R(M,N) \cong \Ext^{n}_R(\Omega M,\Omega N)$; see, for example, \cite[2.6]{KK}.
\end{chunk}

\begin{prop} \label{mt} Let $R$ be a local ring, with $d=\depth(R)$, $M$ be an $R$-module, and let $n\geq 0$ be an integer such that $d\geq n+1$. Assume the following hold:
\begin{enumerate}[\rm(i)]
\item $R$ satisfies $(S_2)$.
\item $M_{\fp}$ is a free $R_{\fp}$-module for all $\fp \in \Spec(R)-\{\fm\}$ and $\depth_R(M)\geq n$. 
\item $\Ext^{d-1}_R(M,M)=0$. 
\item $\Ext^{i}_R(M,R)=0$ for $\min\{d,d-n+2\}\leq i\leq 2d-n+1$.
\end{enumerate}
Then $\pd_R(M)<\infty$. 
\end{prop}

\begin{proof} The case where $d=1$ is trivial since we assume $\Ext^{d-1}_R(M,M)=0$. Hence we assume $d\geq 2$. 

Note, since $d\geq \min\{d,d-n+2\}$, part (iv) implies:
\begin{equation}\tag{\ref{mt}.1}
\Ext^{i}_R(M,R)=0 \text{ for } i = d,\ldots, 2d-n.
\end{equation}
Therefore, in view of (\ref{mt}.1), the observation recorded in \ref{l1} yields:
\begin{equation}\notag{}
\Ext^{d-1}_R(M,M) \twoheadrightarrow \Ext^{d-1}_R(\Omega_R M,\Omega_R M)  \iso \Ext^{d-1}_R(\Omega_R^2 M,\Omega_R^2 M) \iso\cdots\iso \Ext^{d-1}_R(\Omega_R^{d-n+1} M,\Omega_R^{d-n+1} M). 
\end{equation}
As $\Ext^{d-1}_R(M,M)=0$, we conclude that $\Ext^{d-1}_R(\Omega_R^{d-n+1} M,\Omega_R^{d-n+1} M)=0$. We set $X=\Omega_R^{d-n+1} M$. Then it follows that 
\begin{equation}\tag{\ref{mt}.2}
\Ext^{d-1}_R(X,X)=0. 
\end{equation}

We deduce from part (iv) that: 
\begin{equation}\tag{\ref{mt}.3}
\Ext^{i+d-n+1}_R(M,R)\cong \Ext^{i}_R(X, R)=0 \text{ for all } i=1, \ldots, d.
\end{equation}

As $M_{\fp}$ is a free $R_{\fp}$-module for all $\fp \in \Spec(R)-\{\fm\}$ and $\depth_R(M)\geq n$, it follows that $M$ is an $n$th syzygy module; see \cite[2.4]{ArMo}. 
So $X$ is a $(d+1)$st syzygy module. Moreover $X_{\fp}$ is a free $R_{\fp}$-module for all  $\fp \in \Spec(R)-\{\fm\}$ since so is $M_{\fp}$. Thus \cite[Thm. 43]{Vla} implies that 
\begin{equation}\tag{\ref{mt}.4}
\Ext^{i}_R(\Tr_R X, R)=0 \text{ for all } i=1, \ldots, d+1. 
\end{equation}
Finally we make use of \ref{mt}.2), (\ref{mt}.3) and (\ref{mt}.4), and conclude from Kimura's result \ref{KK}(ii) that $X$ is free. Consequently $\pd_R(M)<\infty$. 
\end{proof}

Next is a corollary of Proposition \ref{mt} which establishes a variation of Proposition \ref{corAson}.

\begin{cor} \label{corneww} Let $R$ be a local ring satisfying $(S_2)$, with $d=\depth(R)\geq 1$, and let $M$ be an $R$-module.
Assume the following hold:
\begin{enumerate}[\rm(i)]
\item $M_{\fp}$ is a free $R_{\fp}$-module for all $\fp \in \Spec(R)-\{\fm\}$.
\item $M$ is torsion-free.
\item $\Ext^{d-1}_R(M,M)=0$. 
\end{enumerate}
If $\Ext^{i}_R(M,R)=0$ for all $i=d, \ldots, 2d+1$, then $\pd_R(M)\leq d-2$.
\end{cor}

\begin{proof} Let $\fp$ be a prime ideal of $R$. Then, as $R$ satisfies $(S_1)$, $M_{\fp}$ is torsion-free over $R_{\fp}$. Hence, if $\depth(R_{\fp})\geq 1$, then $\depth_{R_{\fp}}(M_{\fp})\geq 1$. This shows that $M$ satisfies $(\widetilde{S}_1)$. So $\pd_R(M)< \infty$ by Proposition \ref{mt}. Also, since $\depth_R(M)\geq 1$ and $\Ext^{d-1}_R(M,M)=0$, it follows that $\pd_R(M)\leq d-2$; see \ref{Mat}.
\end{proof}

\section*{acknowledgements}
The authors thank Uyen Le and Brian Laverty for their comments on the manuscript; to Le for discussions on the proof of Theorem \ref{thm2}; to Mohsen Asgarzadeh for discussions about \ref{preserve}(b); and to Yongwei Yao for his suggestions about \ref{corFrob} and \ref{preserve}(b). 

Part of this work was completed when Kobayashi visited WVU  in September 2022 and March 2023. He is grateful for the kind hospitality of the WVU School of Mathematical and Data Sciences.

\bibliographystyle{amsplain}

\end{document}